\documentclass[leqno,10pt]{article}
\usepackage{amsmath}
\usepackage{amsfonts}
\usepackage{array}
\usepackage{enumerate}
\usepackage{graphicx}
\usepackage{amssymb}
\usepackage{amsthm}
\usepackage{xcolor}
\usepackage{chemarr}
\usepackage{authblk}

\newtheorem{theorem}{Theorem}
\newtheorem{proposition}{Proposition}
\newtheorem{lemma}{Lemma}

\newtheorem{definition}{Definition}

\newtheorem{remark}{Remark}
\newtheorem{example}{Example}
\newcommand{\R}{\mathbb{R}}

\newcommand{\norm}[2][\relax]{\ifx#1\relax \ensuremath{\left\Vert#2\right\Vert} \else \ensuremath{\left\Vert#2\right\Vert_{#1}}\fi}

\begin{document}

\title{Coordinate-independent singular perturbation reduction for systems with three time scales}

\author{Niclas Kruff, Sebastian Walcher\\
Mathematik A, RWTH Aachen\\
52056 Aachen, Germany
}

\affil{niclas.kruff@matha.rwth-aachen.de, walcher@matha.rwth-aachen.de}
\date{}

\maketitle
\begin{abstract}On the basis of recent work by Cardin and Teixeira on ordinary differential equations with more than two time scales, we devise a coordinate-independent reduction for systems with three time scales; thus no a priori separation of variables into fast, slow etc. is required. Moreover we consider arbitrary parameter dependent systems and extend earlier work on Tikhonov-Fenichel parameter values -- i.e. parameter values from which singularly perturbed systems emanate upon small perturbations -- to the three time-scale setting. We apply our results to two standard systems from biochemistry.

\end{abstract}

\section{Introduction and overview}
Ordinary differential equations involving a small parameter appear frequently in mathematics and in science. Their principal use in chemistry and biochemistry -- which is the main topic of the present paper --  is to find certain (attracting) invariant sets and to achieve reduction of dimension. The mathematical basis is singular perturbation theory, originally due to Tikhonov \cite{tikh} and Fenichel \cite{fenichel}, for systems with one small parameter $\varepsilon$  (or, in other words, for systems with two time scales). \\
While Tikhonov's and Fenichel's theory is concerned with first order approximations in $\varepsilon$, there exist approaches to include higher order terms in $\varepsilon$, e.g. to improve accuracy in the approximation of invariant manifolds; see for instance the critical survey in Kaper and Kaper \cite{kapkap}. More recently,  Noel et al. \cite{noelgvr}, Radulescu et al. \cite{rvg}, Samal et al. \cite{sgfr,sgfwr} developed an algorithmic method to compute slow-fast scenarios in chemical reaction networks, using tropical geometry.  Concerning the existence (or persistence) of  invariant sets obtained by such (a priori formal) calculations one may invoke hyperbolicity properties; for instance Theorem 4.1  in Chicone \cite{chi} is very useful in this respect. A direct method for chemical reaction networks involving different orders of a single small parameter, given certain properties of the system, is due to Cappeletti and Wiuf \cite{cawi}. \\ 
A different perspective is the consideration of systems with more than two time scales by introducing, cum grano salis, several small parameters $\varepsilon_1,\,\varepsilon_2,\ldots$, and to obtain invariant manifolds and reduction on this basis. (One has the option to set all parameters equal in the end.)

Recently Cardin and Teixeira \cite{cartex} generalized Fenichel's fundamental theorems, proving results on invariant sets and reductions of systems with more than two time scales. Here, the differential equation systems are assumed to have variables separated into blocks of fast, slow, ``very slow'' ones, and so on. \\
The present paper is based, on the one hand, on Cardin and Teixeira \cite{cartex}. On the other hand, we extend earlier work \cite{gw2,gwz} that is concerned with coordinate-independent reduction (not requiring an a priori separation of slow and fast variables), as well as with the basic question of finding -- in arbitrary parameter dependent systems -- critical parameter values from which singular perturbation reductions emanate. \\
We will focus on the three time-scale setting, essentially to keep notation manageable, and will only briefly sketch extensions to more than three time scales. Furthermore we will mostly consider systems that satisfy not only the normal hyperbolicity conditions from \cite{cartex} but have the stronger feature of exponential attractivity. One reason for this restriction lies in our interest in chemical reaction networks. But beyond this practical consideration, the algorithms to compute critical parameter values for singular perturbation scenarios indeed requires this additional property.\\
The paper is organized as follows. In Section \ref{secsep} we review the work by Cardin and Teixeira \cite{cartex}. Section \ref{secfree} generalizes the coordinate-independent reduction algorithm from \cite{gw2} to three-timescale systems. In Section \ref{seccrit} we start from a general parameter dependent system and extend the work from \cite{gwz} on critical parameter values (Tikhonov-Fenichel parameter values) to three time scales (resp. two ``small parameters''), and in Section \ref{examplesec} we discuss two classical examples (cooperativity with two complexes, competitive inhibition) from biochemistry in detail. Section \ref{secconc} contains a few remarks about more than three time scales, and finally, for the reader's convenience, we prove some essentially known facts in an Appendix.

\section{Separated fast and slow variables}\label{secsep}
In this section we review and specialize results from Cardin and Teixeira \cite{cartex} for a parameter dependent ordinary differential equation system
\begin{equation}\label{fullsys}
\begin{array}{rcl}
\dot x_1&=&\phantom{\varepsilon_1}\phantom{\varepsilon_2}f_1(x,\varepsilon_1,\varepsilon_2)\\
\dot x_2&=&\phantom{\varepsilon_1}\varepsilon_1f_2(x,\varepsilon_1,\varepsilon_2)\\
\dot x_3&=&\varepsilon_1\varepsilon_2f_3(x,\varepsilon_1,\varepsilon_2)\\
\end{array}; \quad\text{briefly  } \dot x=f(x,\varepsilon_1,\varepsilon_2).
\end{equation}
Here $x=(x_1,x_2,x_3)^{\rm tr}\in\mathbb R^n$ with $x_1\in\mathbb R^{n_1}$, $x_2\in\mathbb R^{n_2}$, and $x_3\in\mathbb R^{n_3}$, and $f$ is smooth on an open neighborhood of $U\times [0,\delta_1)\times[0,\delta_2)$, with $U\subseteq \mathbb R^n$ open and nonempty, and $\delta_1>0$, $\delta_2>0$. \\
We define 
\begin{equation}
{\cal M}_1:=\left\{ x\in U;\, f_1(x,0,0)=0\right\}
\end{equation}
and 
\begin{equation}
{\cal M}_2:=\left\{ x\in U;\, f_1(x,0,0)=f_2(x,0,0)=0\right\},
\end{equation}
and we will assume throughout that these sets are nonempty.
Cardin and Teixeira require some hyperbolicity conditions, which we state here in slightly stronger versions, for the sake of simplicity:
\begin{itemize}
\item {\em First hyperbolicity condition:}  For every $x\in {\cal M}_1$, all the eigenvalues of $D_{x_1}f_1(x,0,0)$\footnote{For a smooth function $g=g(x,y,\ldots)$ we denote the partial derivatives by $D_xg$, $D_yg$ etc. } have nonzero real parts. \\
For sufficiently small $\varepsilon_1,\varepsilon_2$ this condition implies local solvability of the implicit equation $f_1(x,\varepsilon_1,\varepsilon_2)=0$ in the form $x_1=g(x_2,x_3,\varepsilon_1,\varepsilon_2)$, and one may furthermore write 
\[
f_2(x,\varepsilon_1,\varepsilon_2)=\widetilde f_2(x_2,x_3,\varepsilon_1,\varepsilon_2):=
f_2(g(x_2,x_3,\varepsilon_1,\varepsilon_2),x_2,x_3,\varepsilon_1,\varepsilon_2).
\]
\item {\em Second hyperbolicity condition:} For every $x\in{\cal M}_2$, all the eigenvalues of $D_{x_2}\widetilde f_2(x,0,0)$ have nonzero real parts.\footnote{In \cite{cartex} the second hyperbolicity condition is erroneously written for $f_2$ rather than $\widetilde f_2$. The authors are aware of this and will publish a corrigendum.}
\end{itemize}
By suitable choice of $U$, $\delta_1$ and $\delta_2$ we may assume that ${\cal M}_1$ and ${\cal M}_2$ are submanifolds.\\
Continuing to follow \cite{cartex} we introduce the {\em auxiliary system}
\begin{equation}\label{auxsys}
\begin{array}{rcl}
0 &=&\phantom{\varepsilon_2}f_1(x,0,\varepsilon_2)\\
\dot x_2&=&\phantom{\varepsilon_1}f_2(x, 0,\varepsilon_2)\\
\dot x_3&=&\varepsilon_2f_3(x,0,\varepsilon_2)\\
\end{array}
\end{equation}
on 
\[
{\cal M}_2^{\varepsilon_2}:=\left\{ x\in U;\, f_1(x,0,\varepsilon_2)=0\right\},
\]
and the {\em intermediate reduced system}
\begin{equation}\label{interredsys}
\begin{array}{rcl}
0 &=&f_1(x,0,0)\\
\dot x_2&=&f_2(x, 0,0)\\
\dot x_3&=&0\\
\end{array}
\end{equation}
on ${\cal M}_1$.
In both equations \eqref{auxsys} and \eqref{interredsys} above the dot denotes differentiation with respect to $\tau_2:=\varepsilon_1 t$.
By suitable choice of $\delta_2$ we may also assume that every ${\cal M}_2^{\varepsilon_2}$ is a submanifold of $\mathbb R^n$.\\ Finally we define  the {\em completely reduced system}
\begin{equation}\label{compredsys}
\begin{array}{rcl}
0 &=&f_1(x,0,0)\\
0 &=&f_2(x, 0,0)\\
\dot x_3&=&f_3(x,0,0)\\
\end{array}
\end{equation}
on ${\cal M}_2$,
where the dot in \eqref{compredsys} denotes differentiation with respect to $\tau_3:=\varepsilon_1 \varepsilon_2 t$.
\\

We replace the hyperbolicity conditions from \cite{cartex} by stronger requirements, since in our applications we focus on attracting invariant manifolds. 
\begin{definition}\label{hypadef}
We say that system \eqref{fullsys} satisfies the hyperbolic attractivity condition {\em (HA)} if
$D_{x_1}f_1(x,0,0)$ has only eigenvalues with negative real part on $\mathcal M_1$ and if furthermore $D_{x_2}\widetilde f_2(x,0,0)$ has only eigenvalues with negative real part on $\mathcal M_2$.
\end{definition}
Our starting point is the following theorem, specialized from Cardin and Teixeira \cite{cartex}, Theorems A, B and Corollary A. Some of our statements are informal; for rigorous statements and pertinent definitions we refer to \cite{cartex}.
\begin{theorem}\label{splitthm} Let system \eqref{fullsys} be given, with {\em (HA)} satisfied.
\begin{enumerate}[(a)] 
\item Let $\mathcal N\subseteq \mathcal M_2$ be a compact submanifold (with nonempty interior in the relative topology, and possibly with boundary). Then for all sufficiently small $\varepsilon_1, \varepsilon_2$ there exists a locally invariant manifold $\mathcal N_{\varepsilon_1, \varepsilon_2}$ for system \eqref{fullsys} which is $O(\varepsilon_1+ \varepsilon_2)$ close to $\mathcal N$, diffeomorphic to $\mathcal N$ and locally exponentially attracting. Given the appropriate time scales, solutions of \eqref{fullsys} on  $\mathcal N_{\varepsilon_1, \varepsilon_2}$ converge to solutions of \eqref{compredsys} on $\mathcal N$.
\item Let $\varepsilon_2$ be sufficiently small and let $\mathcal L\subseteq \mathcal M_2^{\varepsilon_2}$ be a compact submanifold (with nonempty interior in the relative topology, and possibly with boundary). Then for all sufficiently small $\varepsilon_1$ there exists a locally invariant manifold $\mathcal L_{\varepsilon_1, \varepsilon_2}$ for system \eqref{fullsys} which is $O(\varepsilon_1+ \varepsilon_2)$ close to $\mathcal L$, diffeomorphic to $\mathcal L$ and locally exponentially attracting. Given the appropriate time scales, solutions of \eqref{fullsys} on  $\mathcal L_{\varepsilon_1, \varepsilon_2}$ converge to solutions of \eqref{interredsys} on $\mathcal L$.
\end{enumerate}
\end{theorem}
As given, the part regarding $\widetilde f_2$ in condition (HA) is not ready to use in applications. We provide two equivalent versions.
\begin{proposition}
Condition {\em (HA)} is equivalent to either of the following conditions.
\begin{enumerate}[(i)]
\item $D_{x_1}f_1(x,0,0)$ has only eigenvalues with negative real parts on $\mathcal M_1$, and
\[
B_1(x):= -D_{x_1}f_2(x,0,0)D_{x_1}f_1(x,0,0)^{-1}D_{x_2}f_1(x,0,0)+D_{x_2}f_2(x,0,0)
\]
has only eigenvalues with negative real parts on $\mathcal M_2$.
\item  $D_{x_1}f_1(x,0,0)$ has only eigenvalues with negative real parts on $\mathcal M_1$, and for all sufficiently small $\varepsilon>0$ the matrix
\[
B_2(x,\varepsilon):= \begin{pmatrix} D_{x_1}f_1(x,0,0) & D_{x_2}f_1(x,0,0)\\
                         \varepsilon D_{x_1}f_2(x,0,0) &\varepsilon  D_{x_2}f_2(x,0,0)\end{pmatrix}
\]
has only eigenvalues with negative real parts on $\mathcal M_2$.
\end{enumerate}
\end{proposition}
\begin{proof} We use the notions introduced with the hyperbolicity condition (H) and Definition \ref{hypadef}. From
\[
f_1(g(x_2,x_3,\varepsilon_1,\varepsilon_2),x_2,x_3,\varepsilon_1,\varepsilon_2)=0
\]
one gets by the chain rule
\[
D_{x_2}g(x_2,x_3)= - D_{x_1}f_1(x,,\varepsilon_1,\varepsilon_2)^{-1} D_{x_2}f_1(x,,\varepsilon_1,\varepsilon_2)
\]
when $f_1(g(x_2,x_3,\varepsilon_1,\varepsilon_2),x_2,x_3,\varepsilon_1,\varepsilon_2)=0$, and a further application of the chain rule shows the equivalence of (HA) and (i). The equivalence of (i) and (ii) follows from Lemma \ref{linsing} in the Appendix.\\

\end{proof}
\begin{remark}\label{variationrem}
\begin{enumerate}[(a)]
\item One may rewrite systems \eqref{fullsys} through \eqref{compredsys} to some extent, with no effect on the reductions. Using Hadamard's lemma, one may restate \eqref{fullsys} as
\[
\begin{array}{rclll}
\dot x_1&=&\widehat f_1(x,\varepsilon_2)&+\varepsilon_1 \widehat f_{1,1}(x,\varepsilon_1)&+\varepsilon_1 \varepsilon_2\widehat f_{1,2}(x,\varepsilon_1,\varepsilon_2)\\
\dot x_2&=& & \phantom{+}\varepsilon_1\widehat f_2(x,\varepsilon_1)&+\varepsilon_1 \varepsilon_2\widehat f_{2,2}(x,\varepsilon_1,\varepsilon_2)\\
\dot x_3&=& & &\phantom{+} \varepsilon_1\varepsilon_2\widehat f_3(x,\varepsilon_1,\varepsilon_2)\\
\end{array}
\]
with only the $\widehat f_i$ remaining in the subsequent reductions. Thus the auxiliary system becomes
\[
\begin{array}{rcl}
0 &=&\phantom{\varepsilon_2}\widehat f_1(x,\varepsilon_2)\\
\dot x_2&=&\phantom{\varepsilon_1}\widehat f_2(x, 0)\\
\dot x_3&=&\varepsilon_2f_3(x,0,\varepsilon_2)\\
\end{array}
\]
and there are analogous modifications for the intermediate and the fully reduced system.
\item The passage from \eqref{fullsys} to the completely reduced system \eqref{compredsys} can evidently be obtained in the following manner: Fix $\varepsilon_1>0$ and reduce \eqref{fullsys} with respect to the small parameter $\varepsilon_2$ (in time scale $\varepsilon_2t$). Then let $\varepsilon_1\to 0$, rescaling time once more to $\tau_3$. We will use this observation later on.
\end{enumerate}
\end{remark}
\section{Coordinate-free reduction}\label{secfree}
In the present section we generalize the coordinate-independent reduction procedure from \cite{nw11,gw2} to the three-timescale setting. The first task is to intrinsically characterize those systems which admit a transformation to ``standard form'' \eqref{fullsys}. Reversing matters, applying a (local) smooth coordinate transformation to equation \eqref{fullsys} yields a smooth system
\begin{equation}\label{fullcofree}
\dot x=g^{(0,0)}(x,\varepsilon_1,\varepsilon_2)+\varepsilon_1\left( g^{(1,0)}(x,\varepsilon_1,\varepsilon_2)+\varepsilon_2g^{(1,1)}(x,\varepsilon_1,\varepsilon_2)\right)
\end{equation}
on an open neighborhood of $\widetilde U\times [0,\delta_1)\times[0,\delta_2)\subseteq \mathbb R^n\times\mathbb R\times\mathbb R$ ($\widetilde U\subseteq\mathbb R^n$ open), evidently satisfying the following conditions:
\begin{enumerate}[(i)]
\item For all sufficiently small $\varepsilon_1\geq 0,\,\varepsilon_2\geq 0$, the zeros of $ g^{(0,0)}(x,\varepsilon_1,\varepsilon_2)$ form a submanifold $\widetilde{\mathcal M}_1\subseteq\widetilde U$, of codimension $n_1$, $1\leq n_1<n$.  Given any compact submanifold $\mathcal P_1\subseteq \widetilde{\mathcal M}_1$, there exists $\theta_1>0$ such that at every $y\in \mathcal P_1$ the derivative $D_x g(y,\varepsilon_1,\varepsilon_2)$ admits the eigenvalue zero with algebraic and geometric multiplicity $n-n_1$, and the remaining eigenvalues have real parts $\leq -\theta_1$.
\item For all sufficiently small $\varepsilon_1>0,\,\varepsilon_2\geq 0$ the zeros of 
\[
g^{(0,0)}(x,\varepsilon_1,\varepsilon_2)+\varepsilon_1 g^{(1,0)}(x,\varepsilon_1,\varepsilon_2)
\]
 form a submanifold $\widetilde{\mathcal M_2}\subseteq\widetilde U$, of codimension $n_1+n_2$, $1\leq n_2<n-n_1$. Moreover, for any compact submanifold $\mathcal P_2\subseteq \widetilde{\mathcal M}_2$  there exists a $\theta_2>0$ with the following property: At every $y\in \mathcal P_2$ the derivative $D_xg^{(0,0)}(y,\varepsilon_1,\varepsilon_2)+\varepsilon_1D_xg^{(1,0)}(y,\varepsilon_1,\varepsilon_2)$  admits the eigenvalue zero with algebraic and geometric multiplicity $n-n_1-n_2$, and the remaining eigenvalues have real parts $\leq -\theta_2\varepsilon_1$.
\end{enumerate}
By Remark \ref{variationrem} one may assume that system \eqref{fullcofree} is in the special form
\begin{equation}\label{restcofree}
\dot x=g^{(0,0)}(x,\varepsilon_2)+\varepsilon_1\left(g^{(1,0)}(x,\varepsilon_1)+\varepsilon_2g^{(1,1)}(x)\right)+O(\varepsilon_2(\varepsilon_1+\varepsilon_2)),
\end{equation}
adjusting conditions (i) and (ii) accordingly. Conditions (i) and (ii) are certainly necessary for \eqref{fullcofree} or \eqref{restcofree} to be a transformed version of \eqref{fullsys}. The first part of the next lemma shows sufficiency.
\begin{lemma}\label{decomplem}
\begin{enumerate}[(a)]
\item There exists a local diffeomorphism transforming system \eqref{restcofree} to a system of type \eqref{fullsys} with condition {\em(HA)} if and only if conditions (i) and (ii) above hold.
\item Condition (i) for system \eqref{restcofree} is equivalent to the following: For any $y\in \widetilde{\mathcal M_1}$ there exist a neighborhood $U_{1,y}$, a smooth map $P_1:\,U_{1,y}\to\mathbb R^{n\times n_1}$ such that $P_1(y,\varepsilon_2)$ has rank $n_1$, and a smooth map $\mu_1:\,U_{1,y}\to\mathbb R^{n_1}$ such that $D_x\mu_1(y,\varepsilon_2)$ has rank $n_1$, yielding a decomposition
\[
g^{(0,0)}(x,\varepsilon_2)=P_1(x,\varepsilon_2)\mu_1(x,\varepsilon_2),
\]
and moreover there is a $\theta_1>0$ such that
\[
A_1(x,\varepsilon_2):=D\mu_1(x,\varepsilon_2)P_1(x,\varepsilon_2)
\]
has only eigenvalues with real part $\leq -\theta_1$, for all $x\in U_{1,y}$.
\item In presence of condition (i), condition (ii) for system \eqref{restcofree} is equivalent to the following: For every (sufficiently small) $\varepsilon_1>0$ and any $y\in \widetilde{\mathcal M_2}$ there exist a neighborhood $U_{2,y}$, a smooth map $P_2:\,U_{2,y}\to\mathbb R^{n\times n_2}$ such that $\left(P_1(y,\varepsilon_2), \varepsilon_1 P_2(y,\varepsilon_1)\right)$ has rank $n_1+n_2$, and a smooth map $\mu_2:\,U_{2,y}\to\mathbb R^{n_2}$ such that $\left(D_x\mu_1(y,\varepsilon_2), D_x\mu_2(y,\varepsilon_2)\right)^{\rm tr}$ has rank $n_1+n_2$, yielding a decomposition
\[
g^{(0,0)}(x,\varepsilon_2)+\varepsilon_1 g^{(1,0)}(x,\varepsilon_1)=P_1(x,\varepsilon_2)\mu_1(x,\varepsilon_2)+\varepsilon_1 P_2(x,\varepsilon_1)\mu_2(x,\varepsilon_1),
\]
and moreover there is a $\theta_2>0$ such that
\[
A_2(x,\varepsilon_1,\varepsilon_2):=\begin{pmatrix}D\mu_1(x,\varepsilon_2)   \\ D\mu_2(x,\varepsilon_1) \end{pmatrix} \begin{pmatrix}  P_1(x,\varepsilon_2)&\varepsilon_1 P_2(x,\varepsilon_1)\end{pmatrix}
\]
has only eigenvalues with real part $\leq -\theta_2\varepsilon_1$, for all $x\in U_{2,y}$.
\end{enumerate}
\end{lemma}
\begin{proof}
The nontrivial assertion of part (a) follows from the existence of $n-n_1$ independent first integrals of $g^{(0,0)}$ in a neighborhood of $y$, which was noted by Fenichel \cite{fenichel}, Lemma 5.3 for smooth vector fields, and shown in \cite{nw11}, Proposition 2.2 for the analytic setting, and likewise from the existence of $n-n_1-n_2$ independent first integrals of $g^{(0,0)}+\varepsilon_1g^{(1,0)}$ in a neighborhood of $y$. These first integrals determine slow and ``very slow'' variables. Parts (b) and (c) are straightforward applications of \cite{gw2}, Theorem 1, Remark 4 and Remark 2.
\end{proof}
\begin{remark}\label{ratrem}
The existence of the decomposition $g^{(0,0)}=P_1\,\mu_1$ in part (b) (as well as the decomposition in part (c)) is a consequence of the implicit function theorem in the smooth or analytic case. For polynomial or rational vector fields there exists a decomposition with rational functions as entries of $P_1$ and $\mu_1$, and there is an algorithmic approach to its computation. See \cite{gw2} for details.
\end{remark}
Next we use the decompositions to compute reductions.
\begin{proposition}\label{cofreeredprop}
\begin{enumerate}[(a)]
\item  In arbitrary coordinates the reduction corresponding to the passage from system \eqref{fullsys} to the auxiliary system may be obtained as follows:\\
Given $\varepsilon_2\geq 0$, determine the projection matrix
\[
Q_1(x,\varepsilon_2):=I_n-P_1(x,\varepsilon_2)A_1(x,\varepsilon_2)^{-1}D_x\mu_1(x,\varepsilon_2).
\]
The auxiliary system \eqref{auxsys} for $\varepsilon_2>0$ then corresponds to
\[
\dot x=Q_1(x,\varepsilon_2)\left(g^{(1,0)}(x,0)+\varepsilon_2g^{(1,1)}(x)\right)
\]
on the local invariant manifold defined by $\mu_1(x,\varepsilon_2)=0$.
The equation corresponding to the intermediate reduced system \eqref{interredsys} is obtained by setting $\varepsilon_2=0$.
\item  In arbitrary coordinates the reduction corresponding to the passage from system \eqref{fullsys} to the completely reduced system \eqref{compredsys} may be obtained as follows:\\
Given $\varepsilon_1>0$, determine the projection matrix
\[
\widetilde Q_2(x,\varepsilon_1):=I_n-\begin{pmatrix}P_1(x,0),&\varepsilon_1 P_2(x,0)\end{pmatrix}A_2(x,\varepsilon_1,0)^{-1}\begin{pmatrix}D_x\mu_1(x,0)\\D_x\mu_2(x,0)\end{pmatrix}.
\]
Then $\widetilde Q_2(x,\varepsilon_1)$ extends smoothly to a matrix valued function $Q_2(x)$ at $\varepsilon_1= 0$. The equation corresponding to the completely reduced  system in arbitrary coordinates is given by
\[
\dot x =Q_2(x)\,g^{(1,1)}(x)
\]
on the local invariant manifold defined by $\mu_1(x,0)=\mu_2(x,0)=0$.
\end{enumerate}
\end{proposition}
\begin{proof} Part (a) is a direct application of \cite{gw2}, Theorem 1. For part (b) this theorem is also applicable, but there is a  technical problem involving $\widetilde Q_2$ as $\varepsilon_1\to 0$, since $A_2(x,0)$ is non-invertible. 
To resolve this difficulty, recall that $\widetilde Q_2(x,\varepsilon_1)$ is the projection map onto the kernel of 
\[
D_xg^{(0,0)}(x,0)+\varepsilon_1 D_xg^{(1,0)}(x,\varepsilon_1)
\]
along the image, for $ x\in \widetilde {\mathcal M}_2$ (see \cite{gw2}, Remark 1). With the conditions given in Lemma \ref{decomplem} (c) the image is equal to the column space of $(P_1,\,\varepsilon_1 P_2)$, which in turn equals the column space $W_1$ of $(P_1,\, P_2)$. The latter matrix has full rank at $\varepsilon_1=0$, and its entries depend smoothly on $\varepsilon_1$ and $x$. Moreover the kernel is equal to the kernel
of $(D_x\mu_1,\,D_x\mu_2)^{\rm tr}$, and we may assume w.l.o.g.~that 
\[
\begin{pmatrix}D_x\mu_1\\ D_x\mu_2\end{pmatrix}=\begin{pmatrix} A_1 & A_2\end{pmatrix}
\]
with invertible $A_1$, whence the kernel is equal to the column space $W_2$ of the matrix
\[
\begin{pmatrix}-A_1^{-1}A_2\\ I\end{pmatrix}
\]
with entries depending smoothly on $\varepsilon_1$. Thus there remains to verify that the matrix of the projection onto $W_2$ along $W_1$ depends smoothly on $\varepsilon_1$. For the sake of completeness we give a proof of this fact in Lemma \ref{projlem}, Appendix.
\end{proof}
We note that the reduction also works, including convergence properties, under the weaker assumption corresponding to (H) rather than (AH) for $A_1$ and $A_2$ in Lemma \ref{decomplem}.
\begin{remark}\label{initrem}

While Proposition \ref{cofreeredprop} provides the reduced equations, one also needs initial values for these, which may be obtained from an initial value $y$ of system \ref{restcofree} with the help of the first integrals noted in the proof of Lemma \ref{decomplem}(a); see \cite{gw2}, Proposition 2:
\begin{itemize}
\item Assuming that $y$ is sufficiently close to $\widetilde {\mathcal M}_1$ , the corresponding initial value (up to an error of order $\varepsilon_1+\varepsilon_2$) for the auxiliary system and for the intermediate reduced system is the (locally unique) intersection of $\widetilde {\mathcal M}_1$ and the level sets of $n-n_1$ independent first integrals of $\dot x=g^{(0,0)}(x,0)$ which contain $y$.
\item Assuming that $y$ is sufficiently close to $\widetilde {\mathcal M}_2$ , the corresponding initial value (up to an error of order $\varepsilon_1+\varepsilon_2$) for the auxiliary system and for the intermediate reduced system is the (locally unique) intersection of $\widetilde {\mathcal M}_2$ and the level sets of $n-n_1-n_2$ independent common first integrals of $\dot x=g^{(0,0)}(x,0)$ and $\dot x=g^{(1,0)}(x,0,0)$ which contain $y$. (A direct application of Proposition 2 in \cite{gw2}  would lead to simultaneous first integrals of $\dot x=g^{(0,0)}(x,0)+\varepsilon_1 g^{(1,0)}(x,\varepsilon_1,0)$ for all $\varepsilon_1$. This is equivalent to the condition stated.)
\end{itemize}
\end{remark}
To illustrate the procedure with an example, we recall the competitive inhibition network with substrate $S$, enzyme $E$, inhibitor $I$ and two complexes $C_1,\,C_2$; see for instance Keener and Sneyd \cite{KeSn}. The reaction scheme is given by

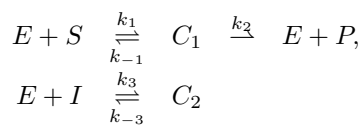
\[
\begin{array}{rcccl}
E+S&\overset{k_1}{\underset{k_{-1}}{\rightleftharpoons}}&
C_1&\overset{k_2}{\rightharpoonup}&E+P,\\
E+I&\overset{k_{3}}{\underset{k_{-3}}{\rightleftharpoons}}&
C_2& & 
\end{array}
\]
which leads (with the usual assumptions of mass action kinetics, spatial homogeneity and constant thermodynamical parameters) to the differential equation system
\begin{equation}\label{compinhib}
  \begin{array}{rcl}
 \dot s&=& k_{-1}c_1-k_1s(e_0-c_1-c_2)\\
 \dot c_1&=&k_1s(e_0-c_1-c_2)-(k_{-1}+k_2)c_1\\
 \dot c_2&=&k_3(e_0-c_1-c_2)(i_0-c_2)-k_{-3}c_2
 \end{array}
\end{equation}
for the concentrations. (The original system is five dimensional; the two linear first integrals $e+c_1+c_2$ and $i+c_2$ yield reduction to dimension three.)

\begin{example}\label{competitiveinhibition}{\em 
In system \eqref{compinhib} set $x=(s,\,c_1,\,c_2)^{\rm tr}$ and assume $k_2=\varepsilon_1\varepsilon_2k_2^*$, $k_3=\varepsilon_1k_3^*$ and $k_{-3}=\varepsilon_1k_{-3}^*$. (Coloquially speaking, binding to the inhibitor and degradation from the inhibitor complex are slow, while degradation from the substrate complex to enzyme and product is very slow.) This is of the type \eqref{restcofree}, with
\begin{align*}
 g^{(0,0)}(x)&=\begin{pmatrix}k_{-1}c_1-k_1s(e_0-c_1-c_2)\\ k_1s(e_0-c_1-c_2)-k_{-1}c_1\\ 0\end{pmatrix},\\
 g^{(1,0)}(x,\varepsilon_1)&=\begin{pmatrix}0\\ 0\\ k_3^*(e_0-c_1-c_2)(i_0-c_2)-k_{-3}^*c_2\end{pmatrix},\\
 g^{(1,1)}(x,\varepsilon_1,\varepsilon_2)&=\begin{pmatrix}0\\-k_2^*c_1\\0\end{pmatrix}.
\end{align*}
Moreover $\widetilde{\mathcal{M}}_2$ is contained in the common zero set of
 \[
\mu_1=k_{-1}c_1-k_1s(e_0-c_1-c_2) \text{  and    }\mu_2= k_3^*(e_0-c_1-c_2)(i_0-c_2)-k_{-3}^*c_2,
\]
$\widetilde{\mathcal{M}}_1$ is contained in the zero set of $\mu_1$, and we have
\[
 P_1(x,\varepsilon_2)=\begin{pmatrix}1\\-1\\0\end{pmatrix},\quad
 P_2(x,\varepsilon_1)=\begin{pmatrix}0 \\ 0 \\ 1\end{pmatrix}.
\]
We determine the auxiliary system and the intermediate reduced system. With 
\[
D\mu_1=\left(-k_1(e_0-c_1-c_2,\,k_1s+k_{-1}, k_1s\right)
\]
one has 
\[
D\mu_1P_1=-k_1(e_0-c_1-c_2-(k_1s+k_{-1})=:-\nu_1
\]
and furthermore
\[
\begin{array}{rcl}
Q_1&=&I_3+\frac{1}{\nu_1}\begin{pmatrix}*&k_1s+k_{-1}&k_1s\\ * &-(k_1s+k_{-1})&-k_1s\\0&0&0\end{pmatrix}\\
       &=&\frac{1}{\nu_1}\begin{pmatrix}*&k_1s+k_{-1}&k_1s\\ * &k_1(e_0-c_1-c_2)&-k_1s\\0&0&\nu_1\end{pmatrix}.
\end{array}
\]
Application to 
\[
g^{(1,0)}+\varepsilon_2g^{(1,1)}=\mu_2\begin{pmatrix}0\\ 0\\ 1\end{pmatrix}-\varepsilon_2k_2^*c_1\begin{pmatrix}0\\ 0\\ 1\end{pmatrix}
\]
yields the auxiliary system  (in time scale $\varepsilon_1 t$) on $\widetilde{\mathcal{M}}_1$:
\[
\begin{pmatrix}\dot s \\ \dot c_1 \\ \dot c_2\end{pmatrix}=\frac{\mu_2}{\nu_1}\begin{pmatrix}k_1s\\ -k_1s\\ \nu_1\end{pmatrix}-\varepsilon_2\frac{k_2^*c_1}{\nu_1}\begin{pmatrix}k_1s+k_{-1}\\ k_1(e_0-c_1-c_2)\\ 0\end{pmatrix}. 
\]
Setting $\varepsilon_2=0$ one obtains the intermediate reduced system.\\
When the initial values for system \eqref{compinhib} are given by $(s_0, c_{1,0}, c_{2,0})$, to obtain the approximate initial values $(s^*_0, c^*_{1,0}, c^*_{2,0})$ on $\widetilde{\mathcal M_1}$ one uses (according to Remark \ref{initrem}) the two first integrals $s+c_1$ and $c_2$ of $g^{(0,0)}$ and the defining equation for $\widetilde{\mathcal M_1}$, thus the system
\[
\begin{array}{rcl}
s+c_1&=&s_0+c_{1,0}\\
c_2&=&c_{2,0}\\
k_{-1}c_1-k_1s(e_0-c_1-c_2)&=&0
\end{array}
\]
which leads to quadratic equations for $s$ and $c_1$.\\

To find the fully reduced system one first computes
\[
D\mu_2=\left( 0 ,\, -k_3^*(i_0-c_2),\,-k_3^*(e_0+i_0-c_1-2c_2) -k_{-3}^*\right)
\]
and
\[
\begin{array}{rcl}
A_2&=&\begin{pmatrix} D\mu_1\\ D\mu_2\end{pmatrix}\begin{pmatrix}P_1,& \varepsilon_1 P_2\end{pmatrix}\\
      &=&\begin{pmatrix} -k_1(e_0-c_1-c_2)-k_1s-k_{-1} &\varepsilon_1\cdot k_1s \\
               k_3^*(i_0-c_2) & -\varepsilon_1\cdot(k_3^*(e_0+i_0-c_1-2c_2) +k_{-3}^*)\end{pmatrix}.
\end{array}
\]
The computation of the projection matrix is straightforward (although a software system is helpful) but the output is sizeable. We just record the
fully reduced system (in time scale $\varepsilon_1\varepsilon_2 t$). It is given by 
\[
\dot x=\frac{1}{\nu_2}\cdot \begin{pmatrix}\xi_1\\\xi_2\\\xi_3\end{pmatrix}
\]
with
\begin{align*}
\nu_2=&sc_1k_1k_3^*+sc_2k_1k_3^*-se_0k_1k_3^*-c_1^2k_1k_3^*-3c_1c_2k_1k_3^*+2c_1e_0k_1k_3^*+c_1i_0k_1k_3^*-2c_2^2k_1k_3^*\\
&+3c_2e_0k_1k_3^*+c_2i_0k_1k_3^*-e_0^2k_1k_3^*-e_0i_0k_1k_3^*-sk_{-3}^*k_1+c_1k_{-3}^*k_1+c_1k_{-1}k_3^*+c_2k_{-3}^*k_1\\
&+2c_2k_{-1}k_3^*-e_0k_{-3}^*k_1-e_0k_{-1}k_3^*-i_0k_{-1}k_3^*-k_{-3}^*k_{-1}
\end{align*}
and
\begin{align*}
\xi_1=&k_2^*(se_0k_{-3}^*k_1+c_1e_0k_{-1}k_3^*-c_1i_0k_{-1}k_3^*+c_2^2k_{-1}k_3^*-c_2e_0k_{-1}k_3^*-c_2i_0k_{-1}k_3^*\\
&+e_0i_0k_{-1}k_3^*-c_2k_{-3}^*k_{-1}),\\[0.2cm]
\xi_2=&\frac{k_1k_2^*}{k_3^*}(c_1^3(k_3^{*})^2-2c_1^2e_0(k_3^{*})^2+2c_1^2i_0(k_3^{*})^2+c_1e_0^2(k_3^{*})^2-2c_1e_0i_0(k_3^{*})^2\\
&+c_1i_0^2(k_3^{*})^2+c_2^3(k_3^{*})^2-c_2^2e_0(k_3^{*})^2-2c_2^2i_0(k_3^{*})^2+2c_2e_0i_0(k_3^{*})^2+c_2i_0^2(k_3^{*})^2-e_0i_0^2(k_3^{*})^2\\
&-c_1^2k_{-3}^*k_3^*+c_1e_0k_{-3}^*k_3^*+2c_1i_0k_{-3}^*k_3^*-3c_2^2k_{-3}^*k_3^*+2c_2e_0k_{-3}^*k_3^*+3c_2i_0k_{-3}^*k_3^*\\
&-2e_0i_0k_{-3}^*k_3^*+2c_2(k_{-3}^{*})^2),\\[0.2cm]
\xi_3=& -\frac{k_{-3}^*k_1k_2^*}{k_3^*}(c_1i_0k_3^*-c_2^2k_3^*+c_2e_0k_3^*+c_2i_0k_3^*-e_0i_0k_3^*+c_2k_{-3}^*)
\end{align*}
restricted to the invariant curve $\widetilde{\mathcal M}_2$.\\
Finally, given initial values $(s_0, c_{1,0}, c_{2,0})$ for system \eqref{compinhib}, approximate initial values for the fully reduced system may be determined by solving the algebraic equations $s+c_1=s_0+c_{1,0}$, $\mu_1=0$ and $\mu_2=0$.
}
\end{example}

\section{Critical parameter values}\label{seccrit}
Typically in applications one starts with a general parameter dependent system, rather than a system of type \eqref{fullsys} or \eqref{fullcofree} with pre-assigned ``small parameters''. Therefore the first task is to determine critical parameter values, for which small perturbations  lead to singular perturbation scenarios. Thus we consider {\em Tikhonov-Fenichel parameter values}, as defined in \cite{gwz} for two time scales, and extend the notion to the three time scale setting. 
\subsection{Tikhonov-Fenichel parameter values}
Tikhonov-Fenichel parameter values (TFPV) were introduced in \cite{gwz} for polynomial (or rational) parameter dependent systems
\begin{equation}\label{arbpar}
\dot x = h(x,\pi),\quad x\in\mathbb R^n,\,\pi\in\Pi\subseteq\mathbb R^m.
\end{equation}
A TFPV $\widehat \pi$ is characterized by the property that small perturbations $\pi=\widehat \pi+\varepsilon \rho+\cdots$ along a smooth curve in parameter space $\Pi$  give rise to a singular perturbation reduction for
\[
\dot x=h(x,\widehat \pi+\varepsilon\rho+\cdots)=h(x,\widehat\pi)+\varepsilon D_\pi h(x,\widehat\pi)\rho+\cdots
\]
with locally exponentially attracting critical manifold. (The definition extends easily to smooth systems but the algorithmic approach relies on the stronger assumption.) There exists an intrinsic characterization of TFPV's, see \cite{gwz} Lemmas 1 and 2, for which the characteristic polynomial
\begin{equation}\label{charpol}
\chi(\tau,x,\pi)=\tau^n+\sigma_{n-1}(x,\pi)\tau^{n-1}+\cdots+ \sigma_1(x,\pi)\tau+\sigma_0(x,\pi)
\end{equation}
of the Jacobian $D_xh(x,\pi)$ is relevant.
We recall:
\begin{lemma}\label{tfpvchar}
Given $0<s<n$, a parameter value $\widehat\pi$ is a TFPV with locally exponentially attracting critical manifold $Z_s$ (depending on $\widehat \pi$) of dimension $s$, and $x_0 \in Z_s$, if and only if the following hold:
\begin{itemize}
\item $h(x_0,\widehat\pi)=0$.
\item The characteristic polynomial $\chi(\tau,x,\pi)$ from \eqref{charpol}) satisfies
\begin{enumerate}[(i)]
\item $\sigma_0(x_0,\widehat\pi)=\cdots=\sigma_{s-1}(x_0,\widehat\pi)=0$;
\item all roots of $\chi(\tau,x_0,\widehat\pi)/\tau^s$ have negative real parts.
\end{enumerate}
\item The system $\dot x=h(x,\widehat\pi)$ admits $s$ independent local analytic first integrals at $x_0$.
\end{itemize}
\end{lemma}
All the conditions in the lemma can be represented by polynomial equations and inequalities. The condition on the roots of $\chi(\tau,x_0,\widehat\pi)/\tau^s$ is characterized by inequalities: There exist $n-s$ Hurwitz determinants (see e.g. Gantmacher \cite{Gant}, Ch.~V, \S6, Thm.~4 ff.) which must attain values $>0$. Moreover, the existence requirement for $s$  independent first integrals leads to a series of polynomial equations via degree by degree evaluation of Taylor expansions. More precisely, for every $d>0$ there is an induced action of $D_xh(x,\pi)$ on the space $S_1+\cdots +S_d$ of polynomials in $n$ variables with zero constant term and of degree $\leq d$.  Extending condition (i), the characteristic polynomial of this action (which coincides with \eqref{charpol} for $d=1$) must have vanishing coefficients for all sufficiently small powers of the indeterminate. (No further inequalities appear, due to the structure of the eigenvalues for this action.) Thanks to Hilbert's {\em Basissatz}, finitely many of these equations suffice. A full account is given in \cite{gwz}.\\

 For the remainder of this section we assume that $\Pi\subseteq \mathbb R_+^m$ is a semi-algebraic set, and that system \eqref{arbpar} admits the positively invariant subset $\mathbb R^n_+$. Then, as was shown in \cite{gwz}, the Tikhonov-Fenichel parameter values for dimension $s$, $1\leq s<n$ form a semi-algebraic subset $\Pi_s\subseteq\mathbb R^m$. We will denote the Zariski closure of $\Pi_s$ by $W_s$. Thus the elements of $W_s$ satisfy all defining equations for $\Pi_s$ but not necessarily the defining inequalities.


\subsection{Nested Tikhonov-Fenichel parameter values}
Generalizing the approach to TFPV in \cite{gwz}), and taking into account the special form of \eqref{fullcofree}, it seems reasonable to consider surfaces in parameter space. Thus  
consider a smooth surface of the special form
\[
 \gamma(\varepsilon_1,\varepsilon_2)=\widehat{\pi}+\varepsilon_1\left({\rho}_1(\varepsilon_1)+\varepsilon_2{\rho_2}(\varepsilon_1,\varepsilon_2)\right)
\]
defined in some nighborhood of $(0,\,0)$.
Substitute $\gamma(\varepsilon_1,\,\varepsilon_2)$ for $\pi$ in \eqref{arbpar} to get

\begin{equation}\label{hadam}
\begin{array}{rcl}
 h\left(x,\gamma(\varepsilon_1,\varepsilon_2)\right)&
 =&\underbrace{h(x,\widehat{\pi})}_{=:g^{(0,0)}}+\underbrace{h\left(x,\gamma(\varepsilon_1,0)\right)-h(x,\widehat{\pi})}_{=:\varepsilon_1\cdot g^{(1,0)}}\\
& +&\underbrace{\left(h(x,\gamma(\varepsilon_1,\varepsilon_2))-h(x,\widehat{\pi})\right)-\left(h(x,\gamma(\varepsilon_1,0))-h(x,\widehat{\pi})\right)}_{=:\varepsilon_1\varepsilon_2g^{(1,1)}}
\end{array}
\end{equation}
with the $g^{(i,j)}$ smooth by Hadamard's lemma. In order to obtain a system \eqref{fullcofree} that also satisfies the conditions (i) and (ii) preceding Lemma \ref{decomplem}, the following is necessary: There exist $s>0$ and $k>0$ such that  $\widehat{\pi}\in \Pi_{s+k}$, and  $\widehat{\pi}+\varepsilon_1\cdot{\rho_1}(\varepsilon_1)\in \Pi_s$ for all sufficiently small $\varepsilon_1>0$. (Note that $\varepsilon_2$ plays no role in these conditions.) This observation gives rise to:

\begin{definition}\label{defnested}
Given system \eqref{arbpar} and $s,\,k>0$ with $s+k<n$, let $\delta>0$ and let
\[\beta:(-\delta,\delta)\longrightarrow {\Pi},\ \varepsilon_1\mapsto \beta(\varepsilon_1)\] be a smooth curve such that 
 \begin{enumerate}[(i)]
  \item $\beta(\varepsilon_1)\in \Pi_s$ for all $\varepsilon_1>0$,
  \item $\widehat{\pi}:=\beta(0)\in \Pi_{s+k}$.
 \end{enumerate}
Then we call $\widehat{\pi}$ a Tikhonov-parameter value (for dimension $s+k$) nested in $\overline{\Pi}_s$.
\end{definition}
We note some properties of nested TFPV.
\begin{proposition}\label{intheboundary}
\begin{enumerate}[(a)]
\item Any TFPV $\widehat{\pi}\in \Pi_{s+k}$ which is nested in $\overline{\Pi}_s$ lies in the boundary of $\Pi_s$ relative to its Zariski closure $W_s$.
\item Let $\beta$ as in Definition \ref{defnested}, and for $\varepsilon_1>0$ consider the decomposition
\[
h(x,\,\beta(\varepsilon_1))=P^*(x,\,\varepsilon_1)\mu^*(x,\,\varepsilon_1)
\]
according to \cite{gw2}, Theorem 1. Then 
\[
\det D\mu^*(x,\,0)\,P^*(x,0))=0
\]
on the critical manifold.
\end{enumerate}
\end{proposition}
\begin{proof} Part (a) is a direct consequence of the definition. As for part (b), at $\varepsilon_1=0$, with $\widehat \pi\in  \Pi_{s+k}$ and $x_0\in  Z_{s+k}$ (using notation from Lemma \ref{tfpvchar}), the coefficient $\sigma_s(x_0,\widehat\pi)$ of the characteristic polynomial \eqref{charpol} of
\[
D_xh(x_0,\,\widehat\pi)=P^*(x_0,0)D\mu^*(x_0,0)
\]
must vanish. This is equivalent to non-invertibility of $D\mu^*(x,\,0)\,P^*(x,0))$; see e.g. \cite{gw2}, Remark 4.

\end{proof}
\begin{remark}\label{boundaryrem} Proposition \ref{intheboundary} opens a starting point for the computation of nested TFPV: Start with system \eqref{arbpar} corresponding to
``generic'' parameter values in $\Pi_s$, i.e. parameter values in the intersection of $\Pi_s$ with an irreducible component of the Zariski closure $W_s$. In order to find nested parameters for higher dimension one only needs to look at the boundary of $\overline{\Pi}_s$, and one can use part (b) in order to obtain necessary conditions. Practically this may be realized by determining the decomposition $P\cdot\mu$ for generic $\pi\in \Pi_s$ and then looking at zeros of $D\mu\cdot P$, with parameters in the boundary. (The boundary may also contain further parameter values in $\Pi_s$.)
\end{remark}
\subsection{Special settings for chemical reaction networks}
For chemical reaction networks (CRN) the parameter region is usually given by $\Pi=\mathbb R_+^m$, thus
\[
\pi=\begin{pmatrix}\pi_1\\ \vdots\\ \pi_m\end{pmatrix}\in \R_+^m,
\]
and for many such systems and given $s$, the irreducible components of $W_s$ are just determined by the vanishing of certain of the $\pi_i$; see e.g. \cite{godiss, gwz,gwz3}. (The underlying reason for this fact is the subject of forthcoming work.) Thus we have, for $\pi$ in a given irreducible component:
\begin{enumerate}[(i)]
 \item Upon relabelling, there is an $\ell$, $0<\ell<m$ such that $\pi_i=0$ for all $i\in \{\ell+1,\cdots m\}$; 
 \item the remaining parameters are nonnegative.
\end{enumerate}
In other words, the intersection of $\Pi_s$ with the given irreducible component of $W_s$ corresponds to some subset of $\overline {\mathbb R_+^\ell}$, with boundary $\overline {\mathbb R_+^\ell}\setminus \mathbb R_+^\ell$. This leads to an obvious case-by-case analysis. Note that boundary points may or may not be contained in $\Pi_s$, but there is no loss in starting with ``generic'' parameter values in the interior $\mathbb R^\ell_{>0}$.
 We look at a particular example.
\begin{example}\label{comptfpv}{\em 
We again consider competitive inhibition; see equation \eqref{compinhib}.  Here the parameters are of the form
\[
 \pi=\begin{pmatrix}e_0\\k_1\\k_{-1}\\k_2\\i_0\\k_3\\k_{-3}\end{pmatrix}\in\mathbb R_+^7.
\]
From \cite{gwz}, Proposition 8 we have the necessary condition $e_0k_1k_{2}k_{-3}=0$ for a TFPV in $\Pi_1$, with each of the four cases (e.g. $e_0=0$ and all other parameter values $\geq 0$) yielding a singular perturbation reduction with attracting one dimensional critical manifold. Hence $W_1$ has four irreducible components. In order to find nested TFPV's for dimension 2 we perform a case-by-case investigation. We only consider one case here; see Section \ref{examplesec} for the remaining ones.\\
For the case $k_2=0$ the system is given by 
  \begin{align*}
   \dot s&=k_{-1}c_1-k_1se\\
   \dot c_1&=k_1se-k_{-1}c_1\\
   \dot c_2&=k_3e i-k_{-3}c_2,
  \end{align*}
where we have used the abbreviations $e=e_0-c_1-c_2$ and $i=i_0-c_2$; note that $e\geq 0$ and $i\geq 0$ by design of \eqref{compinhib}.
By Remark \ref{boundaryrem}, nested TFPV's for dimension two and corresponding points in the critical manifold necesarily satisfy $\det\left(D\mu \cdot P\right)=0$, with 
\begin{align*}
 \mu=\begin{pmatrix}k_{-1}c_1-k_1se\\k_3ei-k_{-3}c_2\end{pmatrix},\ P=\begin{pmatrix}1 & 0 \\ -1 & 0 \\ 0 & 1\end{pmatrix},
\end{align*}
\[
 D\mu=\begin{pmatrix}-k_1e& k_1s+k_{-1} & k_1s \\ 0&-k_3 i & -(k_3i+k_3e+k_{-3})\end{pmatrix}.
\]
Proceeding according to Remark \ref{boundaryrem}, we determine the vanishing set of 
\begin{align*}
 &\det\begin{pmatrix}-(k_1e+k_1s+k_{-1}) & k_1s \\ k_3 i & -(k_3i+k_3e+k_{-3})\end{pmatrix}\\
 =&k_1k_3ie+k_1k_3e^2+k_1k_{-3}e+k_1k_3se+k_1k_{-3}s+k_{-1}k_3 i+k_{-1}k_3e+k_{-1}k_{-3}.
\end{align*}
Since all the variables and parameters are nonnegative, this sum equals zero if and only if every summand vanishes.
In particular, $k_{-1}\cdot k_{-3}$ has to vanish for any nested TFPV. We look at the two ensuing cases.
\begin{enumerate}[(i)]
\item $k_{-1}=0$: Then the remaining conditions are
\[
 k_1k_3 ie=k_1k_3e^2=k_1k_{-3}e=k_1k_3se=k_1k_{-3}s=0.
\]
If $k_1=0$ or $k_3=k_{-3}=0$ we obtain a two dimensional variety of stationary points. Checking the attractivity conditions (HA), one finds that these cases yield nested TFPV. If $e=s=0$ holds then we get $e_0-c_1-c_2=0=s$ which corresponds to a one dimensional variety. In case $e=k_{-3}=0$ we get $c_1=0$ while $c_2$ and  $s$ are arbitrary, thus we have a two dimensional (attracting) variety of stationary points.
\item $k_{-3}=0$: In this case there remains
\[
 k_1k_3ie=k_1k_3e^2=k_1k_3se=k_{-1}k_3 i=k_{-1}k_3e=0.
\]
In view of case (i)  we only have to check $k_3=0$ or $e=i=0$. In both cases we get a variety of dimension two. 
 \end{enumerate}
The case $k_3=k_{-3}=0$ leads to system \eqref{compinhib} with $k_3=\varepsilon_1k_3^*$, $k_{-3}= \varepsilon_1k_{-3}^*$ and $k_2=\varepsilon_1\varepsilon_2k_2^*$, the reduction of which was discussed in Example \ref{competitiveinhibition}.}
\end{example}


\section{Further examples}\label{examplesec}
In this section we continue the discussion of the competitive inhibitor network, to some extent, and furthermore present  a fairly complete investigation of a cooperative system with two complexes, following the strategy outlined in Remark \ref{boundaryrem}. Missing from a complete analysis are some cases concerned with boundary points in $\Pi_1\subseteq W_1$ which themselves belong to $\Pi_1$, as well as certain degenerate cases for $\Pi_2$. Moreover we will not generally record routine calculations to verify conditions such as (HA), and for ease of notation we will frequently use the term ``critical manifold'' for the Zariski closure of this object, without mentioning the inequalities to be satisfied.

\subsection{Competitive inhibition (cont.)}\label{detailsinhibition}
 We continue to investigate the competitive inhibition network; see equation \eqref{compinhib}, Examples \ref{competitiveinhibition} and \ref{comptfpv}. The analysis of TFPV which was started in Example \ref{comptfpv} will be finished here. For $\Pi_1$ there are three remaining cases, viz. $e_0=0$, $k_1=0$ and $k_{-3}=0$.
 \begin{enumerate}[(a)]

 \item For $e_0=0$, the system is given by
\begin{equation}\label{compezero}
   \begin{array}{rcl}
   \dot s&=&(k_1s+k_{-1})c_1+k_1sc_2\\
   \dot c_1&=&-(k_1s+k_{-1}+k_2)c_1-k_1sc_2\\
   \dot c_2&=&-k_3ic_1-(k_3 i+k_{-3})c_2
  \end{array}
\end{equation}
 We only consider the generic case for $\Pi_1$, thus all the remaining parameters are $>0$. Then the (only possible) decomposition $P\cdot\mu$ for the right hand side is given by
 \[
\mu=\begin{pmatrix}c_1\\c_2\end{pmatrix}, \quad  P=\begin{pmatrix}k_1s+k_{-1} & k_1s \\ -(k_1s+k_{-1}+k_2) & -k_1s \\ -k_3 i & -(k_3 i+k_{-3})\end{pmatrix}.
 \]
For nested TFPV, a simple computation yields the necessary condition
\[
0= \det(D\mu\cdot P)=ik_{-1}k_3+ik_2k_3+sk_{-3}k_1+k_{-3}k_{-1}+k_{-3}k_2,
\]
with all terms positive; thus every summand must vanish, and in particular
\[
  \det(D\mu\cdot P)=0\ \Rightarrow\ (k_{-1}+k_2)k_{-3}=0\ \Rightarrow\ k_{-1}=k_2=0\ \text{or}\ k_{-3}=0.
\]
In case $k_{-1}=k_2=0$ system \eqref{compezero} admits a two dimensional variety of stationary points, with $k_1s\not=0$  only if $c_1+c_2=0$. The intersection of this variety with the positive orthant is only one dimensional, thus we do not obtain a two dimensional critical manifold. The cases with $k_1s=0$ translate to $k_1=\varepsilon k_1^*$  for the system with small parameters. (Otherwise the critical manifold would be given by $s=0$, which does not contain the line given by $c_1=c_2=0$.) 
Moreover we have $e_0=\varepsilon_1\varepsilon_2e_0^*$, hence every term $k_1e_0s$ is of the form $\varepsilon_1^2\varepsilon_2\cdot(\cdots)$, and the completely reduced system is necessarily trivial.
Likewise, the case $k_{-3}=0$ in system \eqref{compezero} leads to $k_1s=0$. \\
To summarize, the case $e_0=0$ yields no interesting reductions for the three time scale setting, in marked contrast to the familiar (quasi-steady state) reduction for small initial enzyme concentration with two time scales.
\item Next we consider the system with $k_1=0$, i.e.
\begin{align*}
\dot s&=k_{-1}c_1\\
\dot c_1&=-(k_{-1}+k_2)c_1\\
\dot c_2&=k_3ei-k_{-3}c_2.
\end{align*}
Because of Example \ref{comptfpv}  we may assume that $k_2\neq 0$, which yields $c_1=0$ for stationary points.\\
 In case  $k_{-3}=k_3=0$ we indeed have a two dimensional critical manifold. Turning to small parameters we have $k_1=\varepsilon_1\varepsilon_2k_1^*$, $k_3=\varepsilon_1 k_3^*$ and $k_{-3}=\varepsilon_1k_{-3}^*$, and \eqref{compinhib} becomes
\begin{equation}\label{compkone}
\begin{pmatrix}\dot s\\ \dot c_1\\ \dot c_2\end{pmatrix}=\begin{pmatrix}k_{-1}c_1\\ -(k_{-1}+k_2)c_1\\ 0\end{pmatrix}+\varepsilon_1\begin{pmatrix}0\\ 0\\ k_3^*ei-k_{-3}^*c_2\end{pmatrix}+\varepsilon_1\varepsilon_2k_1^*es\begin{pmatrix}-1\\ 1\\ 0\end{pmatrix}
\end{equation}
We compute the reductions for this case. For the auxiliary system (on the critical variety defined by $c_1=0$) we obtain the decomposition 
\[
\begin{pmatrix}k_{-1}c_1\\ -(k_{-1}+k_2)c_1\\ 0\end{pmatrix}=\underbrace{\begin{pmatrix}k_{-1}\\ -(k_{-1}+k_2)\\ 0\end{pmatrix}}_{P_1} \cdot \underbrace{c_1}_{\mu_1},
\]
and a straightforward computation yields the projection matrix
\[
Q_1=\begin{pmatrix} 1 & k_{-1}/(k_{-1}+k_2) & 0\\
                                0&0&0\\
                                0&0& 1\end{pmatrix}
\]
and the auxiliary system 
\[
\begin{pmatrix}\dot s\\ \dot c_1\\ \dot c_2\end{pmatrix}=\begin{pmatrix}0\\ 0\\ k_3^*ei-k_{-3^*}c_2\end{pmatrix}+\varepsilon_2\frac{k_1^*k_2es}{k_{-1}+k_2}\begin{pmatrix}-1\\ 0\\ 0\end{pmatrix}
\]
on the variety defined by $c_1=0$. The intermediate reduced system is obtained setting $\varepsilon_2=0$.\\ 
Turning to the complete reduction, the decomposition of the ``fast part'' of \eqref{compkone} is given by
\[
\begin{pmatrix}k_{-1}&0\\ -(k_{-1}+k_2)&0\\ 0& \varepsilon_1\end{pmatrix}\begin{pmatrix}c_1\\ \mu_2\end{pmatrix}, \quad \text{with  }  \mu_2=k_3^*ei-k_{-3}^*c_2.
\]
One obtains 
\[
A_2=\begin{pmatrix}-(k_{-1}+k_2) & 0 \\ (k_{-1}+k_2)k_3i & -\varepsilon_1\left(k_3(e+i)+k_{-3}c_2\right)\end{pmatrix}
\]
and may continue as prescribed by Proposition \ref{cofreeredprop}. There are shortcuts, though: First note that the critical manifold is given by $c_1=0$, and $c_2$ constant and equal to the smaller solution $\widetilde c_2$ of the quadratic equation
\[
0=\mu_2(0,c_2)= k_3^*(e_0-c_2)(i_0-c_2)-k_{-3}^*c_2.
\]
The completely reduced system will automatically yield $\dot c_1=0$ and $\dot c_2=0$, hence only the first row of the projection matrix needs to be computed. As the final result of the reduction procedure we get the equation
\[
\dot s=-\frac{k_1^*k_2}{k_{-1}+k_2}(e_0-\widetilde c_2)s,
\]
with the dot denoting differentiation with respect to $\varepsilon_1\varepsilon_2 t$.

\item Finally, we  deal with the case $k_{-3}=0$, which does not automatically yield a one dimensional variety of stationary points. System  \eqref{compinhib} becomes
\begin{align*}
\dot s&=k_{-1}c_1-k_1se\\
\dot c_1&=k_1se-(k_{-1}+k_2)c_1\\
\dot c_2&=k_3ei.
\end{align*}
We may assume that $k_1\not=0$ and $k_2\not=0$, otherwise one would arrive at (non-generic) subcases of previously discussed systems. From this we obtain $c_1=0$ and $es=0$ as necessary conditions. Now $e=0$ and nonnegativity of varaibles imply $e_0=0$; a previously discussed case, therefore every stationary point satisfies $s=0$. If $k_3\not=0$ then $i=0$ forces $c_2=i_0$; the corresponding parameter values are not in $\Pi_1$. So the only case remaining is $k_3=k_{-3}=0$ (very slow binding to the inhibitor, very slow degradation of the inhibitor complex), with system
\begin{align*}
\dot s&=k_{-1}c_1-k_1se\\
\dot c_1&=k_1se-(k_{-1}+k_2)c_1\\
\dot c_2&=0.
\end{align*}
To obtain a two dimensional variety of stationary points one has to check the boundary of $\Pi_1$ for nested TFPV, which splits into four cases. We only discuss the case $k_1=0$ here, thus \eqref{compinhib} with small parameters becomes
\begin{equation}\label{compkthree}
\begin{pmatrix}\dot s\\ \dot c_1\\ \dot c_2\end{pmatrix}=\begin{pmatrix}k_{-1}c_1\\ -(k_{-1}+k_2)c_1\\ 0\end{pmatrix}+\varepsilon_1k_1^*es\begin{pmatrix}-1\\ 1\\ 0\end{pmatrix}+\varepsilon_1\varepsilon_2\begin{pmatrix}0\\ 0\\ k_3^*ei-k_{-3}^*c_2\end{pmatrix}
\end{equation}
The computation of the auxiliary system runs similar to the reduction of \eqref{compkone} and yields
\[
\begin{pmatrix}\dot s\\ \dot c_1\\ \dot c_2\end{pmatrix}=\frac{k_1^*k_2es}{k_{-1}+k_2}\begin{pmatrix}-1\\ 0\\ 0\end{pmatrix}+\varepsilon_2\begin{pmatrix}0\\ 0\\ k_3^*ei-k_{-3^*}c_2\end{pmatrix},
\]
on the invariant variety given by $c_1=0$. Finally, the completely reduced system lives on the variety defined by $c_1=s=0$ (a coordinate subspace), and therefore by \cite{gwz3}, Proposition 5 the reduced system may be directly obtained via ``classical'' QSS reduction; yielding 
\[
\dot c_2=k_3^*(e_0-c_2)(i_0-c_2)-k_{-3}^*c_2.
\]
\end{enumerate}

\subsection{A cooperative system}\label{cooperativesystem}
 In this subsection we study the standard cooperative system involving substrate $S$, two complexes $C_1,C_2$, enzyme $E$ and product $P$. The reaction scheme 
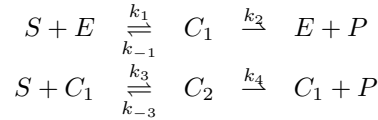
\[
\begin{array}{rcccl}
 S + E &\overset{k_{1}}{\underset{k_{-1}}\rightleftharpoons}& C_1  
&\overset{k_2}{\rightharpoonup}& E+P \\
S+C_1 &\overset{k_{3}}{\underset{k_{-3}}\rightleftharpoons}& C_2 
&\overset{k_4}{\rightharpoonup}& C_1+P\\
\end{array}
\]
yields, with the usual assumptions and stoichiometry, the differential equation
\begin{equation}\label{coopeq}
 \begin{array}{rcl}
  \dot s&=&-k_1e_0s+(k_{-1}+k_1s-k_3s)c_1+(k_1s+k_{-3})c_2\\
  \dot c_1&=&k_1e_0s-(k_{-1}+k_2+k_1s+k_3s)c_1+(k_{-3}+k_4-k_1s)c_2\\
  \dot c_2 &=&k_3 sc_1-(k_{-3}+k_4)c_2
 \end{array}
 \end{equation}
where all appearing constants are non-negative. According to Goeke \cite{godiss}, Kap.\ 9.4, necessary conditions for TFPV  are given by
\[
 e_0k_1k_2(k_{-3}+k_4)=0.
\]
\subsubsection{Case $k_1=0$}
When we substitute $k_1=0$ in equation \eqref{coopeq} we obtain
 \begin{align*}
  \dot s&=(k_{-1}-k_3s)c_1+k_{-3}c_2\\
  \dot c_1&=-(k_{-1}+k_2+k_3s)c_1+(k_{-3}+k_4)c_2\\
  \dot c_2 &=k_3 sc_1-(k_{-3}+k_4)c_2.
 \end{align*}
Hence considering the generic case (all remaining parameters $>0$) we obtain an irreducible component of $W_1$ given by $k_1=0$, and the critical manifold is given by $c_1=c_2=0$. We get a decomposition with
\[
 P=\begin{pmatrix}-sk_3+k_{-1}& k_{-3}\\ -(sk_3+k_{-1}+k_2)&k_{-3}+k_4\\sk_3 &  -(k_{-3}+k_4)\end{pmatrix}, \quad \mu=\begin{pmatrix}c_1\\ c_2\end{pmatrix},
\]
and necessary conditions for nested TFPV from
\[
0=\det D\mu\cdot P=(k_{-1}+k_2)(k_{-3}+k_4)\Rightarrow k_{-1}=k_2=0\text{  or  }k_{-3}=k_4=0.
\]
Thie first set of conditions does not, by itself, yield a two dimensional critical manifold, and we will not pursue it further here. The second set, i.e.\ $k_{-3}=k_4=0$, yields the two dimensional variety given by $ c_1=0$. \\
Considering this setting, we introduce the small parameters in our original system by substituting $k_1=\varepsilon_1\varepsilon_2 k_1^*$, $k_{-3}=\varepsilon_1k_{-3}^*$, $k_4=\varepsilon_1 k_4^*$ . Ordering the parameters as $e_0,k_1,k_{-1},k_2,,k_3,k_{-3},k_4$ , we thus consider the surface in  parameter space given by 
\[
\gamma(\varepsilon_1,\varepsilon_2)=\begin{pmatrix}e_0\\0\\k_{-1}\\k_2\\k_3\\0\\0\end{pmatrix}+\varepsilon_1\cdot \left(\begin{pmatrix}0\\0\\0\\0\\0\\k_{-3}^*\\k_4^{*}\end{pmatrix}+\varepsilon_2\begin{pmatrix}0\\k_1^*\\0\\0\\0\\0\\0\end{pmatrix}\right),
\]
and with $x=(s,\,c_1,\,c_2)^{\rm tr}$ we get
\begin{equation}\label{coopkonezero}
 h(x,\varepsilon_1,\varepsilon_2)=g^{(0,0)}(x)+\varepsilon_1\cdot \left(g^{(1,0)}(x,\varepsilon_1)+\varepsilon_2\cdot g^{(1,1)}(x,\varepsilon_1,\varepsilon_2)\right)
\end{equation}
with
\begin{align*}
 g^{(0,0)}(x)&=\begin{pmatrix}(-sk_3+k_{-1})c_1\\ -(sk_3+k_{-1}+k_2)c_1\\k_3 sc_1\end{pmatrix}\\
 g^{(1,0)}(x,\varepsilon_1)&=\begin{pmatrix}k_{-3}^*c_2\\ (k_{-3}^{*}+k_4^*)c_2\\ -(k_{-3}^*+k_4^*)c_2\end{pmatrix}\\
 g^{(1,1)}(x,\varepsilon_1,\varepsilon_2)&=\begin{pmatrix}sk_1^*c_1+sk_1^*c_2-k_1^*e_0 s\\-(sk_1^*c_1+sk_1^*c_2-k_1^*e_0 s)\\ 0\end{pmatrix}.
\end{align*}
For this system we compute the complete reduction on $c_1=c_2=0$ and the intermediate reduction on $c_1=0$.
In order to compute the completely reduced system, a factorization of $g^{(0,0)}+\varepsilon_1 g^{(1,0)}$  is given by 
\[
\begin{pmatrix}P_1, &\varepsilon_1P_2\end{pmatrix}\cdot\begin{pmatrix}\mu_1\\ \mu_2\end{pmatrix}
\]
with $\mu_1=c_1$, $\mu_2=c_2$, and
\[
 P_1=\begin{pmatrix}-sk_3+k_{-1}\\ -(sk_3+k_{-1}+k_2)\\k_3 s\end{pmatrix},\quad 
 P_2=\begin{pmatrix}k_{-3}^*\\ k_{-3}^{*}+k_4^*\\ -(k_{-3}^*+k_4^*)\end{pmatrix}.
\]
The projection matrix is
\begin{align*}
Q_2=\begin{pmatrix}1 & -\frac{-sk_3k_4^*-k_{-3}^*k_{-1}-k_{-1}k_4^*}{k_{-3}^*k_{-1}+k_2k_{-3}^*+k_{-1}k_4^*+k_2k_4^*} & -\frac{-sk_3k_4^*-2k_{-3}^*k_{-1}-k_2k_{-3}^*-k_{-1}k_4^*}{k_{-3}^*k_{-1}+k_2k_{-3}^*+k_{-1}k_4^*+k_2k_4^*}\\ 0 & 0 & 0\\ 0 & 0 & 0 \end{pmatrix},
\end{align*}
and the fully reduced system in very slow time on the invariant manifold $c_1=c_2=0$ is given by the equation
\[
\dot s=-\frac{k_3k_4^*s+k_{-3}^*k_2+k_4^*k_2}{k_{-3}^*k_{-1}+k_2k_{-3}^*+k_{-1}k_4^*+k_2k_4^*}\cdot k_1^*e_0s.
\]
Similarly one computes the intermediate system on the two dimensional variety given by $c_1=0$ from the decomposition $P_1\cdot\mu_1$:
\begin{align*}
 \begin{pmatrix}\dot s \\ \dot c_1\\ \dot c_2\end{pmatrix}=\frac1{sk_3^*+k_{-1}+k_2}\begin{pmatrix}-(sc_2k_3k_4^*+2k_{-1}c_2k_{-3}^*+k_2k_{-3}^*c_2+c_2k_4^*k_{-1})\\0
\\-(k_{-1}c_2k_{-3}^*+k_2k_{-3}^*c_2+c_2k_4^*k_{-1}+c_2k_2k_4^*)\end{pmatrix}.
\end{align*}
\ \\
\subsubsection{Case $e_0=0$} From $e_0=0$ one also obtains a component of $W_1$, and system \eqref{coopeq} specializes to
 \begin{align*}
 \dot s&=(k_{-1}+k_1 s-k_3 s)c_1+(k_1 s+k_{-3})c_2\\
 \dot c_1&=-(k_{-1}+k_2+k_1s+k_3s)c_1+(k_{-3}+k_4-k_1s)c_2\\
 \dot c_2&=k_3 s c_1-(k_{-3}+k_4)c_2.
 \end{align*}
The right hand side has a factorization $P\cdot \mu$ with
\[
 \mu=\begin{pmatrix}c_1\\c_2\end{pmatrix},\quad\ P=\begin{pmatrix}sk_1-sk_3+k_{-1} & sk_1+k_{-3}\\ -sk_1-sk_3-k_{-1}-k_2 & -sk_1+k_{-3}+k_4\\ sk_3 & -k_{-3}-k_4\end{pmatrix},
\]
and in order to obtain nested TFPV we examine all variable-parameter configurations that satisfy
\begin{align*}
0= \det(D\mu\cdot P)=sk_1\cdot (sk_3+k_{-3}+k_4)+(k_{-1}+k_2)\cdot (k_{-3}+k_4).
\end{align*}
The plane given by $s=0$ is not a viable candidate for a two dimensional critical manifold since it does not contain the line $c_1=c_2=0$. This leaves the cases
 $k_1=k_{-1}=k_2=0$, $k_1=k_{-3}=k_4=0$ and $k_3=k_{-3}=k_4=0$. \\
The first of these yields a two dimensional variety (defined by $k_3sc_1-k_{-3}c_2=0$) only under the additional condition $k_4=0$. The second case, whenever $k_1\not=0$, yields a variety whose intersection with the positive orthant has dimension one, hence is of no relevance. For the third case we obtain a two dimensional variety only if $k_1=0$ or $k_2=0$.\\ With the exception of this very last case, the completely reduced system will always be trivial, due to $k_1=\varepsilon_1k_1^*$ and $e_0=\varepsilon_1\varepsilon_2 e_0^*$, which implies $k_1e_0=O(\varepsilon_1^2\varepsilon_2)$. 
We consider one spacial case, viz.\ the intermediate reduction coresponding to the nested TFPV with $k_1=k_3=k_{-3}=k_4=0$; here $c_1=0$ defines the two dimensional critical manifold. Considering 
\[
 \gamma\left(\varepsilon_1,\varepsilon_2\right)=\begin{pmatrix}0\\0\\k_{-1}\\k_2\\0\\0\\0\end{pmatrix}+\varepsilon_1\cdot \left(\begin{pmatrix}0\\k_1^*\\0\\0\\k_3^*\\k_{-3}^*\\k_4^{*}\end{pmatrix}+\varepsilon_2\begin{pmatrix}e_0^*\\0\\0\\0\\0\\0\\0\end{pmatrix}\right),
\]
 we compute:
\begin{align*}
 g^{(0,0)}&=\begin{pmatrix}c_1k_{-1}\\ -(k_{-1}+k_2)c_1 \\0\end{pmatrix}\\
 g^{(1,0)}&=\begin{pmatrix}(sk_1^*-sk_3^*)c_1+(sk_1^*+k_{-3}^*)c_2\\ -(sk_1^*+sk_3^*)c_1+(-sk_1^*+k_{-3}^*+k_4^*)c_2\\ k_3^*sc_1-(k_{-3}^*+k_4^*)c_2\end{pmatrix}\\
 g^{(1,1)}&=\begin{pmatrix}-\varepsilon_1k_1^*e_0^*s\\ \varepsilon_1k_1^*e_0^*s\\ 0\end{pmatrix}.
\end{align*}
The intermediate reduced system on the invariant variety $c_1=0$ is then:
\begin{align*}
 \begin{pmatrix}\dot s\\ \dot c_1\\ \dot c_2\end{pmatrix}=\begin{pmatrix} \left(s c_2k_1^*k_2+2c_2k_{-3}^*k_{-1}+c_2k_{-3}^*k_2+c_2k_{-1}k_4^*\right)(k_{-1}+k_2)\\0\\-(k_{-3}^*+k_4^*)c_2\end{pmatrix}.
\end{align*}

\subsubsection{Case $k_{-3}=k_4=0$} These conditions define a component of $W_1$, and generically the critical manifold is given by $s=c_1=0$. System \eqref{coopeq} is given by
\[
\begin{array}{rcl}
  \dot s&=&-k_1e_0s+(k_{-1}+k_1s-k_3s)c_1+k_1sc_2\\
  \dot c_1&=&k_1e_0s-(k_{-1}+k_2+k_1s+k_3s)c_1-k_1sc_2\\
  \dot c_2 &=&k_3 sc_1
 \end{array}
\]
and the product decomposition (which we do not write down here) yields
\[
0=\det D\mu\cdot P=k_1(e_0-c_2)(k_2+2k_3s).
\]
as necessary conditions for nested TFPV.
One possible case is $k_1=0$ with critical manifold $c_1=0$. The remaining cases are:
\begin{enumerate}[(i)]
\item $k_2=k_{-1}=0$ with variety $s=0$;
\item $k_2=k_3=0$ with variety $k_1(e_0-c_1-c_2)s-k_{-1}c_1=0$.
\end{enumerate}
Note that the condition $e_0-c_2=0$ does not yield a two dimensional critical variety.
\subsubsection{Case $k_2=0$} In this situation system \eqref{coopeq} simplifies to
\[
 \begin{array}{rcl}
  \dot s&=&-k_1e_0s+(k_{-1}+k_1s-k_3s)c_1+(k_1s+k_{-3})c_2\\
  \dot c_1&=&k_1e_0s-(k_{-1}+k_1s+k_3s)c_1+(k_{-3}+k_4-k_1s)c_2\\
  \dot c_2 &=&k_3 sc_1-(k_{-3}+k_4)c_2.
 \end{array}
\]
The condition $k_2=0$ by itself does not define an irreducible component of $W_1$; in other words it does not guarantee the existence of a one dimensional variety of stationary points. Therefore we first investigate sufficient conditions, using the observation $\dot s+\dot c_1+2\dot c_2=-k_4c_2$.
\begin{enumerate}[(a)]
\item For $k_4\not=0$ this observation implies that any stationary point satisfies $c_2=0$, and the remaining condition is $k_3sc_1=0$. 
\begin{enumerate}[(i)]
\item In case $k_3\not=0$ we have either $s=0$, with the variety of stationary points given by $s=c_2=0$; in turn this yields the parameter configuration
\[
k_{-1}=k_2=0.
\]
\item Alternatively we have $c_1=0$, the variety is given by $c_1=c_2=0$, and one must have $k_1e_0=0$. We obtain the possible parameter configurations
\[
k_1=k_2=0 \text{  or  } e_0=k_2=0.
\]
\end{enumerate}
\item In case $k_4=0$ the remaining system is
\[
\begin{array}{rcl}
\dot s&=& -k_1es+k_{-1}c_1-k_3sc_1+k_{-3}c_2\\
\dot c_1&=& k_1es-k_{-1}c_1+k_3sc_1+k_{-3}c_2\\
\dot c_2&=&k_3sc_1-k_{-3}c_2.
\end{array}
\]
Adding the first two equations for stationary points shows that $k_{-3}c_2=0$, and combining this with the third equation yields $k_3sc_1=0$; in addition one has $k_1es-k_{-1}c_1=0$. Thus there are further conditions for the existence of a one dimensional critical variety.
\begin{enumerate}[(i)]
\item Given that $k_3\not=0$ and $k_{-3}\not=0$, the variety is given either by $c_2=s=0$, which yields the parameter conditions
\[
k_2=k_{-1}=k_4=0;
\]
 or the variety is given by $c_1=c_2=0$, with parameter conditions
\[
k_2=k_4=k_1=0\text{  or  } k_2=k_4=e_0=0;
\]
all of these are special cases from (a).
\item In case $k_3=0$ we obtain the one dimensional variety given by $c_2=0$ and $k_1(e_0-c_1)s-k_{-1}c_1=0$; thus we have the parameter condition
\[
k_2=k_3=k_4=0
\]
which defines a component of $W_1$.
\item In case $k_{-3}=0$ one obtains the variety $s=c_1=0$, with parameter conditions
\[
k_2=k_4=k_{-3}=0.
\]
\end{enumerate}
\end{enumerate}

For all these parameters the next task is to discuss conditions for embedded TFPV. We will only do so for two cases.
\begin{enumerate}
\item In the case $k_{-1}=k_2=0$ one has a decomposition
\[
\begin{pmatrix}-k_1e-k_3c_1 & k_{-3}\\ k_1e-k_3c_1 & k_{-3}+k_4\\k_3c_1 &-( k_{-3}+k_4)\end{pmatrix}\cdot\begin{pmatrix} s\\ c_2\end{pmatrix}
\]
which yields
\[
\det D\mu\cdot P=k_1(k_{-3}+k_4)e+ k_3k_4c_1.
\]
We take a closer look at the case $k_1=k_3=0$, with critical variety $c_2=0$.
The surface in parameter space
\[
 \gamma\left(\varepsilon_1,\varepsilon_2\right)=\begin{pmatrix}e_0\\0\\0\\0\\0\\k_{-3}\\k_4\end{pmatrix}+\varepsilon_1\cdot \left(\begin{pmatrix}0\\k_1^*\\0\\0\\k_3^*\\0\\0\end{pmatrix}+\varepsilon_2\begin{pmatrix}0\\0\\k_{-1}^*\\k_2^*\\0\\0\\0\end{pmatrix}\right)
\]
yields
\begin{align*}
 g^{(0,0)}&=\begin{pmatrix}c_2k_{-3}\\(k_{-3}+k_4)c_2\\-(k_{-3}+k_4)c_2
 \end{pmatrix}\\
 g^{(1,0)}&=\begin{pmatrix}-k_1^*e_0s+(sk_1^*-sk_3^*)c_1+sk_1^*c_2\\ k_1^*e_0s-(k_1^*s+k_3^*s)c_1-k_1^*sc_2\\k_3^*sc_1\end{pmatrix}\\
 g^{(1,1)}&=\begin{pmatrix}k_{-1}^*c_1\\-(k_{-1}^*+k_2^*)c_1\\0\end{pmatrix}.
\end{align*}
The intermediate reduced system is as follows:
\begin{align*}
 \dot s&=\frac{sc_1k_{-3}k_1^*+k_4k_1^*c_1s-sc_1k_3^*k_4-k_1^*e_0sk_{-3}-k_1^*e_0sk_4}{k_{-3}+k_4}\\
 \dot c_1&=-sc_1k_1^*+k_1^*e_0s\\
 \dot c_2&=0,
\end{align*}
and the completely reduced system (on $s=c_2=0$) is given by:
\[
 \dot c_1=\frac{-c_1(c_1k_{-3}k_1^*k_2^*-c_1k_{-1}^*k_3^*k_4+c_1k_1^*k_2^*k_4-c_1k_2^*k_3^*k_4-e_0k_{-3}k_1^*k_2^*-e_0k_1^*k_2^*k_4)}{c_1k_{-3}k_1^*+c_1k_1^*k_4-c_1k_3^*k_4-e_0k_1^*k_{-3}-e_0k_1^*k_4}
\]

\item In the case $k_2=k_3=k_4=0$ we have the one dimensional critical manifold
\[
 k_1s\cdot (e_0-c_1)-k_{-1}c_1=0,\ c_2=0
\]
and the right hand side of the system at $k_2=k_3=k_4=0$ can be decomposed into $P\cdot \mu$, with
\begin{align*}
 P&=\begin{pmatrix}1 & sk_1+k_{-3}\\ -1 & -sk_1+k_{-3}\\ 0 & -k_{-3}\end{pmatrix}\\
 \mu&=\begin{pmatrix}k_1s\cdot (e_0-c_1)-k_{-1}c_1 \\ c_2\end{pmatrix}.
\end{align*}
This yields
\[
\det(D\mu\cdot P)=(k_1(e_0-c_1)+k_1s+k_{-4})\cdot k_{-3}.
\]
We investigate the case $k_{-3}=0$. Additionally setting $k_{-3}=0$ we obtain the two dimensional critical manifold defined by 
\[
\mu_2:= -k_1e_0s+(k_{-1}+k_1s)c_1+k_1sc_2=0.
\]
Following the usual procedure we consider the surface
\[
 \gamma\left(\varepsilon_1,\varepsilon_2\right)=\begin{pmatrix}e_0\\k_1\\k_{-1}\\0\\0\\0\\0\end{pmatrix}+\varepsilon_1\cdot \left(\begin{pmatrix}0\\0\\0\\0\\0\\k_{-3}^*\\0\end{pmatrix}+\varepsilon_2\begin{pmatrix}0\\0\\0\\k_2^*\\k_3^*\\0\\k_4^*\end{pmatrix}\right)
\]
in parameter space, and thus
\begin{align*}
 g^{(0,0)}&=\begin{pmatrix}-k_1e_0s+(sk_1+k_{-1})c_1+k_1sc_2\\k_1e_0s-(sk_1+k_{-1})c_1+-k_1sc_2\\0\end{pmatrix}\\
 g^{(1,0)}&=\begin{pmatrix}k_{-3}^*c_2\\k_{-3}^*c_2\\-k_{-3}^*c_2\end{pmatrix}\\
 g^{(1,1)}&=\begin{pmatrix}-sk_3^*c_1\\ -(sk_3^*+k_2^*)c_1+k_4^*c_2\\ sk_3^*c_1-k_4^*c_2\end{pmatrix}.
\end{align*}
Here the intermediate reduced system is given by
\begin{align*}
\dot s&= \frac{1}{sk_1+k_1(e_0-c_1-c_2)+k_{-1}}\cdot \left(sc_2k_{-3}^*k_1+2c_2k_{-3}^*k_{-1}\right)\\
\dot c_1&=\frac{1}{sk_1+k_1(e_0-c_1-c_2)+k_{-1}}\cdot \left(sc_2k_{-3}^*k_1-2k_1c_2k_{-3}^*c_1-2k_1k_{-3}^*c_2^2+2e_0k_1k_{-3}^*c_2\right)\\
\dot c_2&=-c_2k_{-3}^*
\end{align*}
on $\mu_2=0$, and the fully reduced system is given by
\begin{align*}
\dot s&= \frac{1}{sk_1+k_1(e_0-c_1)+k_{-1}}\cdot \left(-se_0k_1k_2^*\right)\\
\dot c_1&=\frac{1}{sk_1+k_1(e_0-c_1)+k_{-1}}\cdot \left(k_1k_2^*c_1^2-k_2^*e_0k_1c_1\right)\\
\dot c_2&=0.
\end{align*}

\end{enumerate}
\section{More time scales}\label{secconc}
In this section we give a brief outline on extending the coordinate-free approach to more than three time scales. Thus let $N\geq 3$ and first consider a system with $N-1$ small parameters of the form
\begin{equation}\label{fullersys}
\dot x_i = \left(\prod_{1\leq j<i}\varepsilon_j\right)\cdot f_i(x,\varepsilon_1,\ldots,\varepsilon_{N-1}),\quad 1\leq i\leq N; \quad\text{  briefly  }\dot x=f(x,\varepsilon)
\end{equation}
with separated variables. By a smooth coordinate transformation this becomes
\begin{equation}\label{fullercofree}
\dot x=g^{(0,\ldots,0)}+\varepsilon_1\left( g^{(1,0,\ldots0)}+\varepsilon_2\left(g^{(1,1,0,\ldots,0)}+\varepsilon_3\left(\cdots\right)\right)\right)
\end{equation}
with the very last term in the embedded brackets being $\varepsilon_{N-1}g^{(1,\ldots,1)}$. Here all $g^{(i_1,\ldots,i_{N-1})}$ are functions of $(x,\varepsilon_1,\ldots,\varepsilon_{N-1})$.
Moreover conditions (i), (ii)  preceding Lemma \ref{decomplem} generalize in an obvious manner to the vanishing sets of
\[
\begin{array}{rcl}
g^{(0,\ldots,0)}& & \\
g^{(0,\ldots,0)}&+&\varepsilon_1 g^{(1,0,\ldots0)}\\
g^{(0,\ldots,0)}&+&\varepsilon_1\left( g^{(1,0,\ldots0)}+\varepsilon_2g^{(1,1,0,\ldots,0)}\right)\\
 &\text{etc.}&
\end{array}
\]
and as in Proposition \ref{cofreeredprop} one obtains decompositions
\[
\begin{array}{rcl}
g^{(0,\ldots,0)}&=&P_1\mu_1\\
g^{(0,\ldots,0)}+\varepsilon_1 g^{(1,0,\ldots0)}&=&\begin{pmatrix}P_1,&\varepsilon_1 P_2\end{pmatrix}\begin{pmatrix}\mu_1\\ \mu_2\end{pmatrix}\\
g^{(0,\ldots,0)}+\varepsilon_1\left( g^{(1,0,\ldots0)}+\varepsilon_2g^{(1,1,0,\ldots,0)}\right)&=&\begin{pmatrix}P_1,&\varepsilon_1 P_2,&\varepsilon_1\varepsilon_2P_3\end{pmatrix}\begin{pmatrix}\mu_1\\ \mu_2\\ \mu_3\end{pmatrix}\\
                    &\text{etc.}&
\end{array}
\]
Likewise, one generalizes the definitions of  $A_1,\,A_2$ and the constructions of the projection matrices $Q_j$, the latter extending smoothly to $\varepsilon_1=\cdots\varepsilon_{N-1}=0$. This yields the various (intermediate) reductions.\\

Given a general parameter dependent system \eqref{arbpar}, nested Tikhonov-Fenichel parameter values may be found via the ansatz
\[
\gamma:\,(\varepsilon_1,\ldots,\varepsilon_{N-1})\mapsto \widehat \pi+\varepsilon_1\left(\rho_1(x,\varepsilon_1)+\varepsilon_2\left(\rho_2(x,\varepsilon_1,\varepsilon_2)+\varepsilon_3\left(\cdots\right)\right)\right)
\]
and the ensuing decomposition of $h(x,\gamma(\varepsilon_1,\ldots,\varepsilon_{N-1}))$ analogous to the one in \eqref{hadam}. Thus the problem is to find $s>0$, $0<k_1<\cdots <k_{N-1}$ with $s+k_{N-1}<n$ and a smooth map 
\[
\beta:\,(\varepsilon_1,\ldots,\varepsilon_{N-2})\to \Pi,
\]
defined in some neighborhood of $0$, such that 
\[
\begin{array}{rcl}
\beta(\varepsilon_1,\ldots,\varepsilon_{N-2})&\in&\Pi_s\\
\beta(\varepsilon_1,\ldots,\varepsilon_{N-3},0)&\in&\Pi_{s+k_1}\\
                                      &\vdots&  \\
\beta(0,\ldots,0)&\in&\Pi_{s+k_{N-1}}\\
\end{array}
\]
whenever all $\varepsilon_j>0$. Rather obvious generalizations of Proposition \ref{intheboundary} hold, and the strategy outlined in Remark \ref{boundaryrem} remains applicable.
\section*{Appendix}
For the reader's convenience we state and prove here two lemmas.
\begin{lemma}\label{linsing}
Let $r_1,\,r_2$ be positive integers and
\[
A\in\mathbb R^{r_1\times r_1},\quad B\in \mathbb R^{r_1\times r_2},\quad C\in \mathbb R^{r_2\times r_1},\quad D\in \mathbb R^{r_2\times r_2}.
\]
Moreover assume that all eigenvalues of $A$ have real part $<0$. Then the following are equivalent.
\begin{enumerate}[(i)]
\item All eigenvalues of $-CA^{-1}B+D$ have real part $<0$.
\item There exists $\delta>0$ such that, for every $\varepsilon\in (0,\,\delta)$, all eigenvalues of
\[
\begin{pmatrix} A & B\\ \varepsilon C & \varepsilon D\end{pmatrix}
\]
have real part $<0$.
\end{enumerate}
\end{lemma}
\begin{proof}Consider the singularly perturbed linear differential equation
\[
\begin{array}{rcrcr}
\dot x&=& Ax&+&By\\
\dot y&=& \varepsilon Cx&+&\varepsilon Dy
\end{array}
\]
Introducing $z:=x+A^{-1}By$ one can rewrite this as
\[
\begin{array}{rcl}
\dot z &=& Az +\varepsilon (\cdots)\\
\dot y &=&\varepsilon\left((-CA^{-1}B+D) y+ Cz\right)
\end{array}
\]
Here the fast system is just given by $\dot z=Az$, and the slow system (on the critical manifold defined by $z=0$) is given by
\[
\dot y =\varepsilon(-CA^{-1}B+D) y.
\]
Using Tikhonov's theorem (in the form given e.g.\ in Verhulst \cite{verhulst}, Ch.~8), one sees that both conditions (i), (ii) are equivalent to exponential attractivity of the stationary point $0$ for the linear system.
\end{proof}
\begin{lemma}\label{projlem}
Let $V\subseteq \mathbb R^n$ be open and nonempty, $0<r<n$, $\delta>0$ and  
\[
\begin{array}{rll}
B_1: &V \times [0,\,\delta)\to \mathbb R^{n\times r},& (x,\varepsilon)\mapsto B_1(x,\varepsilon)\\
B_2: &V \times [0,\,\delta)\to \mathbb R^{n\times (n-r)},& (x,\varepsilon)\mapsto B_2(x,\varepsilon)\\
\end{array}
\]
be smooth functions (defined in some neighborhood of $V \times [0,\,\delta)$) such that $\mathbb R^n$ is the sum of the image $W_1$ of $B_1$ and the image $W_2$ of $B_2$,  for every $(x,\varepsilon)$. Then the entries of the matrix $Q(x,\varepsilon)\in\mathbb R^{n\times n}$ which sends $v\in\mathbb R^n$ to its $W_2$-component with respect to the direct sum decomposition $W_1\oplus W_2$ depend smoothly on $(x,\varepsilon)$. 
\end{lemma}
\begin{proof}
We suppress the arguments $(x,\varepsilon)$ in the notation. By assumption $C:=\begin{pmatrix} B_1,&B_2\end{pmatrix}$ is invertible, and the entries of $C^{-1}$ depend smoothly on $(x,\varepsilon)$. With the projection matrix given by
\[
Q=\begin{pmatrix}0&B_2\end{pmatrix}C^{-1},
\]
the assertion is obvious.

\end{proof}

\noindent{\bf Acknowledgement.} The work of both authors has been supported by the bilateral project ANR-17-CE40-0036 and DFG-391322026 SYMBIONT. 


\begin{thebibliography}{99}
\bibitem{chi} C.~Chicone: {\it Ordinary differential equations with applications. Second edition}. Texts in Applied Mathematics {\bf 34}, Springer, New York (2006).
\bibitem{cawi}D.~ Capelletti,  C.~Wiuf: {\it Uniform approximation of solutions by elimination of intermediate species in deterministic reaction networks.}  SIAM J. Appl. Dyn. Syst. {\bf 16}, 2259 - 2286 (2017).
\bibitem{cartex} P.T.~Cardin, M.A.~Texeira: {\it Fenichel theory for multiple time scale singular perturbation problems.} SIAM J. Appl. Dyn. Sys. {\bf 16}, 1452-1452 (2017)

\bibitem{fenichel} N. Fenichel:  {\it Geometric singular perturbation theory for ordinary
              differential equations}.  {J. Differential Equations} {\bf 31}(1), 53--98 (1979).

\bibitem{Gant} F.R.~Gantmacher: {\it Applications of the theory of matrices.} Dover,  Mineola (2005).

\bibitem{godiss} A. Goeke: {\it Reduktion und asymptotische Reduktion von Reaktionsglei\-chungen.} Doctoral dissertation, RWTH Aachen (2013). URL: \\{ http://darwin.bth.rwth-aachen.de/opus3/volltexte/2013/4814/pdf/4814.pdf}

\bibitem{gw2} A. Goeke, S. Walcher: {\it A constructive approach to quasi-steady state reduction.} J. Math. Chem. {\bf 52}, 2596 - 2626 (2014).

\bibitem{gwz}  A. Goeke, S. Walcher, E.~Zerz: {\it Determining ``small parameters" for quasi-steady state.} J. Diff. Equations {\bf 259}, 1149--1180 (2015).

\bibitem{gwz3} A. Goeke, S. Walcher, E.~Zerz: {\it Classical quasi-steady state reduction -- A mathematical characterization.} Physica D {\bf 345}, 11 - 26 (2017).

\bibitem{kapkap} H.G.~Kaper, T.J.~Kaper: {\it  Asymptotic analysis of two reduction methods for systems of chemical reactions.} Physica D {\bf 165},  66 - 93 (2002).

\bibitem{KeSn}  J.~Keener, J.~Sneyd: {\it Mathematical physiology I: Cellular physiology}, {Second Ed}.
{Springer-Verlag}, {New York} (2009).

\bibitem{lawa}C.~Lax, S.~Walcher: {\it Singular perturbations and scaling.} To appear in Discrete Contin. Dyn. Syst. Ser. B. {\tt http://arxiv.org/abs/1807.03107} (2018).

\bibitem{noelgvr} V.~Noel, D.~Grigoriev, S.~Vakulenko, O.~Radulescu: {\it Tropicalization and tropical equilibrium of chemical reactions.} In: G.L.~Litvinov, S.N.~Sergeev (eds): {Tropical and idempotent mathematics and applications.} Contemporary Math. {\bf 616}, pp.~261 - 275. Amer. Math. Soc., Providence (2014).

\bibitem{nw11}L.~Noethen, S.~Walcher: {\it Tikhonov's theorem and quasi-steady state.} Discrete Contin. Dyn. Syst. Ser. B {\bf 16}(3), 945--961 (2011).

\bibitem{rvg} O.~Radulescu, S.~Vakulenko, D.~Grigoriev: {\it Model reduction of biochemical reactions networks by tropical analysis methods.} Math. Model. Nat. Phenom. {\bf 10}, 124--138 (2015).

\bibitem{sgfr} S.S.~Samal, D.~Grigoriev, H.~Fr\"ohlich, O.~Radulescu: {\it  Analysis of reaction network systems using tropical geometry.} In: V.P.~Gerdt, W.~Koepf, W.M.~Seiler, E.V.~Vorozhtsov (eds.): {\it Computer Algebra in Scientific Computing. $17^{\rm th}$ International Workshop, CASC 2015.} Lecture Notes in Computer Science {\bf 9301}, Springer-Verlag, Cham (2015), pp. 424--439.

\bibitem{sgfwr} S.S.~Samal, D.~Grigoriev, H.~Fr\"ohlich, A.~Weber, O.~Radulescu: {\it  A geometric method for model reduction of biochemical networks with polynomial rate functions.} Bull. Math. Biol., DOI 10.1007/s11538-015-0118-0 (2015).

\bibitem{tikh} A.N. Tikhonov: {\it Systems of differential equations containing a small parameter multiplying the derivative} (in Russian).
{Math. Sb.} {\bf31}, 575--586 (1952).

\bibitem{verhulst}
F. Verhulst: {\it Methods and Applications of Singular Perturbations. Boundary Layers and Multiple Timescale Dynamics},
Springer, New York (2005).



\end{thebibliography}
\end{document}